\documentclass[11pt,leqno]{amsart}

\usepackage{hyperref}
\usepackage{mathrsfs}
\usepackage{tikz-cd}
\usepackage{amsmath,amsfonts,amssymb}
\usepackage{xfrac}
\usepackage{nicefrac}
\usepackage{color}

\theoremstyle{plain}
\newtheorem{theorem}{Theorem}[section]
\newtheorem{lemma}[theorem]{Lemma}

\theoremstyle{definition}

\newtheorem{remark}[theorem]{Remark}
\newtheorem{example}[theorem]{Example}

\numberwithin{equation}{section}

\DeclareMathOperator{\st}{St}
\DeclareMathOperator{\Aut}{Aut}
\DeclareMathOperator{\St}{St}

\DeclareMathOperator{\Z}{\mathbb{Z}}

\title[Beauville structures for quotients of generalised GGS-groups]{Beauville structures for quotients\\ of generalised GGS-groups}

 \author[E. Di Domenico]{Elena Di Domenico}
 
 \address{Elena Di Domenico: Department of Mathematics, University of Trento, 38123 Trento, Italy - University of the Basque Country UPV/EHU, 48080 Bilbao, Spain}
 \email{elena.didomenico@yahoo.it}
 
 \author[\c{S}. G\"{u}l]{\c{S}\"{u}kran G\"{u}l}
 
\address{\c{S}\"{u}kran G\"{u}l: Department of Mathematics\\ TED University\\
06420 Ankara, Turkey} 
\email{sukran.gul@tedu.edu.tr}
 
 \author[A. Thillaisundaram]{Anitha Thillaisundaram}
 
 \address{Anitha Thillaisundaram: Centre for Mathematical Sciences, Lund University, 223 62 Lund, Sweden}
 \email{anitha.thillaisundaram@math.lu.se}
 
 \keywords{Groups acting on rooted trees, finite $p$-groups, Beauville structures}
 \subjclass[2010]{Primary  20E08;  Secondary 20D15, 14J29}

 \date{\today}
 
\begin{document}

\begin{abstract} A finite group with a Beauville structure gives rise to a certain compact complex surface called a Beauville surface.
  G\"{u}l and Uria-Albizuri showed that quotients of the periodic Grigorchuk-Gupta-Sidki (GGS-)groups that act on the $p$-adic tree, for $p$ an odd prime, admit Beauville structures. We extend their result by showing that quotients of infinite periodic GGS-groups acting on $p^n$-adic trees, for $p$ any prime and $n\ge 2$, also admit Beauville structures.
\end{abstract}

\maketitle

\section{Introduction}

A \emph{Beauville surface} is a compact complex surface isomorphic to $(C_1\times C_2)/G$, where 
\begin{enumerate}
\item[$\bullet$] $C_1$ and $C_2$ are algebraic curves of genus at least 2, and $G$ is a finite group acting freely on $C_1\times C_2$ by holomorphic transformations,
\item[$\bullet$] the group $G$ acts faithfully on the curves $C_i$ such that $C_i/G\cong \mathbb{P}_1(\mathbb{C})$ and the covering map $C_i\rightarrow C_i/G$ is ramified over three points, for $i\in\{1,2\}$.
\end{enumerate}
 The group $G$ is then said to be a \emph{Beauville group}. 
 
 One can reformulate the condition for a finite group~$G$ to be a Beauville group purely in group-theoretical terms: for $x,y\in G$, let
 \[
 \Sigma(x,y)=\bigcup_{g\in G} \big(\langle x\rangle^g \cup \langle y\rangle^g \cup \langle xy\rangle^g\big),
 \]
 that is, the union of all conjugates of the cyclic subgroups generated by $x$, $y$ and $xy$. Then $G$ is a Beauville group if and only if $G$ is 2-generated and there exist generating sets $\{x_1,y_1\}$ and $\{x_2,y_2\}$ of~$G$ such that $\Sigma(x_1,y_1)\cap \Sigma(x_2,y_2)=\{1\}$.
The sets $\{x_1,y_1\}$ and $\{x_2,y_2\}$ are then called a \emph{Beauville structure} for~$G$.

Beauville groups have been intensely studied in recent times; see surveys \cite{ BBF, fai, fai2, jon}. 
For example, the abelian Beauville groups were classified by Catanese \cite{cat}: a finite abelian group $G$ is a Beauville group if and only if $G\cong C_n \times C_n$ for $n> 1$ with $\gcd(n, 6) = 1$. After abelian groups, the most natural class of finite groups to consider are nilpotent groups. The determination of nilpotent Beauville groups is easily reduced to the case of $p$-groups.

In~\cite{BBF}, it was shown that there are non-abelian Beauville $p$-groups of order $p^n$ for every $p\geq5$ and every $n\geq3$.
The first explicit infinite family of Beauville $2$-groups was constructed in \cite{BBPV}. 
In \cite{SV}, Stix and Vdovina constructed an infinite family of Beauville $p$-groups, for every prime $p$, by considering quotients of ordinary triangle groups.
In~\cite{FAG}, Fern{\'a}ndez-Alcober and G{\"u}l extended Catanese's  criterion for abelian Beauville groups to finite $p$-groups satisfying certain conditions which are much weaker than commutativity. They also give the first explicit infinite family of Beauville $3$-groups, and they show that there are Beauville $3$-groups of order $3^n$ for every $n \geq 5$.

In~\cite{GUA},  G\"{u}l and Uria-Albizuri showed that the quotients of periodic Grigorchuk-Gupta-Sidki (GGS-)groups admit Beauville  structures, giving another infinite family of Beauville $p$-groups, for~$p$ an odd prime. 
The GGS-groups were some of the early examples of groups acting on rooted trees, which first arose as easily describable examples of infinite finitely generated periodic groups.  Since then, groups acting on rooted trees have provided many other interesting and exotic examples, such as infinite finitely generated groups of intermediate word growth, infinite finitely generated amenable but not elementary amenable groups, etc; see for instance~\cite{growth, Handbook}.

The GGS-groups are infinite groups acting faithfully on the $p$-adic tree. The natural  quotients of a GGS-group~$G$ are $G/\st_G(n)$, for $n\in\mathbb{N}$, where $\st_G(n)$ is the normal subgroup of the elements of~$G$ that pointwise fix the vertices at the $n$th level of the tree. The quotient $G/\st_G(n)$ acts on the finite tree consisting of the first $n$ layers of the full $p$-adic tree.

We extend the result of~\cite{GUA} to a generalisation of the GGS-groups to the $p^n$-adic tree, for any prime~$p$ and $n\ge 2$, by showing that infinitely many quotients of infinite periodic GGS-groups acting on the $p^n$-adic  tree do admit Beauville structures.

 Briefly, a GGS-group acting on the $p^n$-adic tree is a group $G=\langle a,b\rangle$, where $a=(1\,2\, \cdots \,p^n)$ permutes the first-level vertices, whereas $b$ fixes the first-level vertices and is defined recursively by $b=(a^{e_1},\ldots,a^{e_{p^n-1}},b)$, for $e_1,\ldots,e_{p^n-1}\in \mathbb{Z}/p^n\mathbb{Z}$ with not all $e_i$ being zero. Here $(a^{e_1},\ldots,a^{e_{p^n-1}},b)$ refers to the respective independent actions at the $p^n$ maximal subtrees. We write $\mathbf{e}=(e_1,\ldots,e_{p^n-1})$ and call it the \emph{defining vector} of~$G$. We refer the reader to Section~\ref{sec:preliminaries} for  background material and precise definitions.
These groups were constructed by Vovkivsky~\cite{vov}. A prominent group in this family, that was defined earlier in~\cite{GR}, is the second Grigorchuk group, which acts on the $4$-adic tree with defining vector $\mathbf{e}=(1,0,1)$.

For an infinite periodic GGS-group~$G$ acting on the $p^n$-adic tree, we associate to~$G$ a number $m_G\in\mathbb{N}$ that is related to the order of certain group elements; we refer the reader to  Section~\ref{sec:Beauville} for precise details.
The main result of this paper is as follows.

\begin{theorem}\label{thm:main}
Let $G$ be an infinite periodic GGS-group acting on the $p^n$-adic tree for $p$ a prime and $n\ge 2$.  Then $G/\st_G(k)$ admits a Beauville structure for $k\ge m_G$.
\end{theorem}

 In particular, these quotients of  infinite periodic GGS-groups acting on the $2^n$-adic tree yield infinite families of
 Beauville $2$-groups,  and infinite periodic GGS-groups acting on the $3^n$-adic tree give many more infinite families of Beauville $3$-groups. 
 Note that Beauville $2$-groups and $3$-groups are somewhat rare, precisely because of the small number of maximal subgroups. 
 Indeed, for a time it was unclear whether such groups even existed, until the first examples were given in~\cite{FGJ}.
 We also want to emphasise that for the GGS-groups acting on the $p$-adic tree, if one were to consider the case $p=2$ there is only one group, which is the infinite dihedral group, hence its quotients do not admit Beauville structures, and for $p=3$ only one out of the three isomorphism classes of such groups has quotients admitting Beauville structures; see~\cite{GUA} and~\cite{Petschick}.

For finite 2-groups, one can further ask the question of whether the group gives rise to  a mixed Beauville surface, where a mixed Beauville surface is a Beauville surface where the action
of the group interchanges the two curves $C_1$ and $C_2$ defining the surface. It is quite difficult to find examples of finite 2-groups with mixed Beauville structures; see~\cite{FP} and references therein.
As clarified in Remark~\ref{rmk:mixed}, for $G$ an infinite periodic GGS-group acting on the $2^n$-adic tree for $n\ge 2$, the quotient $G/\st_G(k)$ does not admit a mixed Beauville structure for $k\ge m_G$.

G\"{u}l and Uria-Albizuri also showed in~\cite{GUA} that for $G$ a GGS-group acting on the $p$-adic tree, there are quotients of~$G$ admitting Beauville structures if and only if $G$ is periodic. The situation is not so clear cut for GGS-groups acting on the $p^n$-adic tree, for $n\ge 2$. In particular, there are non-periodic GGS-groups acting on the $p^n$-adic tree which have quotients admitting Beauville structures; see Remark~\ref{rem:non-periodic}. Furthermore, our main result above deals only with infinite periodic GGS-groups. For finite GGS-groups, some quotients admit Beauville structures; see Remark~\ref{rmk:finite}. However a large class~$\mathcal{E}$ of finite GGS-groups have no level-stabiliser quotients admitting Beauville structures. 
We refer the reader to  Section~\ref{sec:exceptional} for the definition of~$\mathcal{E}$.
We have the following result.

\begin{theorem}\label{thm:exceptional}
Let $G\in\mathcal{E}$ be a finite GGS-group acting on the $p^n$-adic tree for $p$ a prime and $n\ge 2$.  Then $G/\st_G(k)$ is not a Beauville group for all $k\in \mathbb{N}$.
\end{theorem}

 To prove our theorems, we largely make use of the subgroup structure of the respective groups.

\smallskip

We note that a generalisation of a Beauville group is a group admitting a ramification structure. These are groups  associated to surfaces isogenous to a higher product of curves. We refer the reader to~\cite{NT} for an overview of groups with ramification structures. Only recently were some groups acting on rooted trees shown to admit (non-Beauville) ramification structures. These include 
quotients of the Grigorchuk groups and quotients of periodic multi-EGS groups acting on the $p$-adic tree, for $p$ an odd prime;  see~\cite{NT} and~\cite{DGT2} respectively.

\smallskip
\noindent\emph{Organisation.} Section~\ref{sec:preliminaries} of this paper consists of background material for groups acting on rooted trees and 
here we also recall the generalisation of the GGS-groups. 
In Section~\ref{sec:properties} we establish key properties of these groups. In Section~\ref{sec:Beauville} we prove Theorem~\ref{thm:main} and lastly in Section~\ref{sec:exceptional} we prove Theorem~\ref{thm:exceptional}.


\section{Preliminaries}\label{sec:preliminaries}

All trees considered here will be rooted, meaning that there is a distinguished vertex called the root, with one degree less than all other vertices. For $d\in\mathbb{N}_{\ge 2}$, let $T$ be the $d$-adic  tree,
  meaning all vertices have $d$ children.  Using the
  alphabet $A = \{1,\ldots,d\}$, the vertices~$u_\omega$ of~$T$ are
  labelled bijectively by elements~$\omega$ of the free
  monoid~$A^*$ as follows: the root of~$T$
  is labelled by the empty word~$\varnothing$, and for each word
  $\omega \in A^*$ and letter $a \in A$ there is an edge
  connecting $u_\omega$ to~$u_{\omega a}$.  We say
  that $u_\omega$ precedes $u_\lambda$ whenever $\omega$ is a prefix of $\lambda$.  When convenient, we do not differentiate between $A^*$ and vertices of~$T$.

  There is a natural length function on~$A^*$. The words
  $\omega$ of length $\lvert \omega \rvert = n$, representing vertices
  $u_\omega$ that are at distance $n$ from the root, are the $n$th-level vertices and form the \textit{$n$th layer} of the tree. The elements of the boundary $\partial T$ correspond
  naturally to infinite simple rooted paths, and they are in one-to-one correspondence with the $d$-adic integers.

  Denote by $T_u$ the full  subtree of $T$ that has its root at
  a vertex~$u$ and includes all vertices succeeding~$u$.  For any
  two vertices $u = u_\omega$ and $v = u_\lambda$, the map
  $u_{\omega \tau} \mapsto u_{\lambda \tau}$, induced by replacing the
  prefix $\omega$ by $\lambda$, yields an isomorphism between the
  subtrees $T_u$ and~$T_v$. 

  Every automorphism of $T$ fixes the root, and the orbits of
  $\mathrm{Aut}(T)$ on the vertices of the tree $T$ are precisely its
  layers. For $f \in \mathrm{Aut}(T)$, the image of a vertex $u$ under
  $f$ is denoted by~$u^f$.  The automorphism $f$ induces a faithful action
  on the monoid~$A^*$ and for
  $a \in A$ we have $(\omega a)^f = \omega^f a'$ where $a'\in A$ is
  uniquely determined by $\omega$ and~$f$.  From this we have a permutation
  $f(\omega)$ of~$A$ where
  \[
  (\omega a)^f = \omega^f a^{f(\omega)}, \qquad \text{and hence}
  \quad   (u_{\omega a })^f = u_{\omega^f a^{f(\omega)}},
  \]
 and $f(\omega)$ is called the \textit{label} of~$f$ at~$\omega$.  The  collection  of  all  labels  of $f$ constitutes the \emph{portrait} of~$f$, and there is a one-to-one correspondence between automorphisms of~$T$ and portraits. 
  The automorphism~$f$ is \textit{rooted} if $f(\omega) = 1$ for
  $\omega \ne \varnothing$.  It is \textit{directed}, with directed
  path~$\ell$ for some $\ell \in \partial T$, if the support
  $\{ \omega \mid f(\omega) \ne 1 \}$ of its labelling is infinite and
  marks only vertices at distance $1$ from the set of
    vertices corresponding  to the path~$\ell$. Additionally, the \textit{section} of~$f$ at a vertex~$u$ is defined to be the unique automorphism $f_u$ of $T \cong T_{u}$ given by the condition $(uv)^f = u^f v^{f_u}$ for $v \in A^*$.


\subsection{Subgroups of $\Aut(T)$}
 For $G\le \Aut(T)$, the
\textit{vertex stabiliser} $\text{st}_G(u)$ is the subgroup
consisting of elements in $G$ that fix the vertex~$u$. For
$n \in \mathbb{N}$, the \textit{$n$th-level stabiliser}
  $\St_G(n)= \bigcap_{\lvert \omega \rvert =n}
  \text{st}_G(u_\omega)$
is the subgroup consisting of automorphisms that fix all vertices at
level~$n$.  Let $T_{[n]}$ be the finite subtree of $T$ of
vertices up to level~$n$. Then  $\St_G(n)$ is the kernel of the induced action of $G$ on $T_{[n]}$, and we will denote by $G_n$ the quotient $G/\St_G(n)$.

Each $g\in \St_{\mathrm{Aut}(T)} (n)$ can be 
  described completely in terms of its restrictions to the subtrees
  rooted at vertices at level~$n$.  Indeed, the map 
\[
\psi_n \colon \St_{\mathrm{Aut}(T)}(n) \longrightarrow
\prod_{\lvert \omega \rvert = n} \mathrm{Aut}(T_{u_\omega})
\cong \mathrm{Aut}(T) \times \overset{d^n}{\cdots} \times
\mathrm{Aut}(T)
\]
is a natural
  isomorphism mapping $g\in \St_{\mathrm{Aut}(T)} (n)$ to its $n$th-level sections. For ease of notation, we write $\psi=\psi_1$.

For  $\omega\in A^*$, we further define: 
\begin{align*}
\varphi_\omega :\text{st}_{\mathrm{Aut}(T)}(u_{\omega}) &\longrightarrow \mathrm{Aut}(T_{u_\omega}) \cong \mathrm{Aut}(T)\\
 g\quad&\longmapsto \quad g_\omega
\end{align*}

A group $G \leq \Aut(T)$ is said to be \emph{self-similar} if for every $g\in G$ and every vertex $u$, the section~$g_u$, of $g$ at $u$, is an element of $G$.
For a self-similar group $G\le \text{Aut}(T)$, for $m\ge 2$ and $i\in A^{m-1}$, we will write
\[
\psi_m^i=\psi\circ\varphi_i.
\]
That is, for $g\in \text{st}_G(u_i)$ with $\varphi_i(g)\in\st_G(1)$, we have $\psi_m^i(g)$ gives the components of $\varphi_i(g)$ corresponding to the children of the $(m-1)$st-level vertex $i$. 
In addition,  we write 
\[
{\phi}_{n,m}:\text{St}_{G_n}(m)\longrightarrow G_{n-1}\times \overset{d^{m}}\cdots\times G_{n-1}
\]
for the corresponding $\psi_m$ when working in the quotient. Likewise, we write $\phi_n=\phi_{n,1}$ for simplicity.


\subsection{GGS-groups acting on the $p^n$-adic tree}\label{sec: GGS-p^n}

Let $T$ be the $p^n$-adic tree for a prime~$p$ and $n\in\mathbb{N}$. Given a non-zero vector $\mathbf{e}=(e_1,\ldots,e_{p^n-1})\in (\Z/p^n\Z)^{p^n-1}$, the GGS-group~$G=G_{\mathbf{e}}$ associated to the defining vector~$\mathbf{e}$  is the group generated by the rooted automorphism~$a$ corresponding to the cycle $(1\, 2\,\cdots\, p^n)$ and by a directed automorphism $b\in \St_{\Aut(T)}(1)$ defined recursively via
\[
\psi(b)=(a^{e_1}, a^{e_2},\ldots, a^{e_{p^n-1}},b).
\]

Unlike the GGS-groups acting on the $p$-adic tree, the GGS-groups acting on the $p^n$-adic tree for $n\ge 2$ are not all infinite. A necessary and sufficient condition for these groups to be infinite was given by Vovkivsky~\cite{vov}. He proved that such a group is  infinite if and only if there exists an $i\geq 0$ such that
\[
R_0\le R_1\le \cdots \le R_i=R_{i+1}=\cdots < n
\]
where the sequence $R_j$ is defined recursively as follows: $R_0$ is the largest integer such that $p^{R_0}\mid e_\ell$ for all $\ell\in\{1,\ldots,p^n-1\}$; and then for $j\geq0$ and while $R_j < n$, the number $R_{j+1}$ is defined as the largest
integer such that $p^{R_{j+1}}$ divides $e_\ell$ for all $\ell\in \{p^{R_j}, 2p^{R_j},\ldots,p^n-p^{R_j}\}$.

Note that the order of $a$ is $p^n$ and the order of $b$ is $p^{n-R_0}$.  Further, from \cite[Thm.~2.1]{Elena-paper}
we have $G/G'\cong C_{p^n} \times C_{p^{n-R_0}}$.

Vovkivsky further proved in~\cite{vov} that  a GGS-group acting on the $p^n$-adic tree is a periodic group if and only if for each $k\in \{0,\ldots,n-1\}$, 
\[
S[k]\equiv 0\pmod{p^{k+1}}
\]
where 
\begin{equation}\label{eq:S}
S[k]=e_{p^{k}}+e_{2p^{k}}+\cdots+e_{p^{n}-p^{k}}.
\end{equation}

We recall the prominent example of such a GGS-group, 
which is for the case $p=n=2$:

\begin{example} Let $T$ be the $4$-adic tree. The \emph{second Grigorchuk group} $\Gamma \leq \Aut(T)$  is generated by two automorphisms $a$ and $b$, where $a$ is the rooted automorphism corresponding to the cycle $(1 \, 2 \, 3 \, 4)$, and $b \in \St_{\Gamma}(1)$ is recursively defined by $\psi(b)=(a,1,a,b)$. The group~$\Gamma$,  which is periodic, was first defined in~\cite{GR}. 
\end{example}

For more information on GGS-groups acting on the $p^n$-adic tree, see~\cite{vov,Elena,Elena-paper}.


\section{Properties of GGS-groups acting on the $p^n$-adic tree}\label{sec:properties}

For~$G=G_{\mathbf{e}}$, a GGS-group acting on the $p^n$-adic tree,
recall that $R_0$ is the largest integer such that $p^{R_0}\mid e_\ell$ for all $\ell\in\{1,\ldots,p^n-1\}$. 

Suppose now that $G$ is periodic. To each element $a^{ip^r}b^{jp^s}$ in~$G$, where $0\le r< n$, $0\le s< n-R_0$ and $i,j\not\equiv 0 \pmod p$, we associate the following numbers:
\begin{enumerate}
\item[(i)] If $p^{n-s}\nmid S[r]$, let  $ t_0,t_1,\ldots, t_{m_{r,s}}\in\mathbb{N}$ for some $1\le m_{r,s}< n$ be  such that  
\begin{enumerate}
\item[$\bullet$] $t_{k}<n-s\le t_{m_{r,s}}$ for all $k<m_{r,s}$,
\item[$\bullet$] $p^{t_{0}}\mid S[r]$ and $p^{t_{0}+1}\nmid S[r]$, 
\item[$\bullet$] $p^{t_k} \mid S[t_{k-1}+s]$ and $p^{t_k+1} \nmid S[t_{k-1}+s]$ for $1\le k< m_{r,s}$,
\item[$\bullet$] $p^{t_{m_{r,s}}}\mid S[t_{m_{r,s}-1}+s]$,
\end{enumerate}
where the $S[\lambda]$ for $\lambda\in \{0,1,\ldots,n-1\}$ are as in~(\ref{eq:S}).
\item[(ii)] If $p^{n-s}\mid S[r]$, then we set $m_{r,s}=0$,  and for convenience define $t_0=n-s$.
\end{enumerate}

If we are in case (i), note that $t_0$ does not depend on $s$.

\begin{lemma}\label{orders}
Let $G=\langle a,b\rangle$ be a periodic GGS-group acting on the $p^n$-adic  tree, let $0\le r< n$, $0\le s< n-R_0$  and $i,j\not\equiv 0 \pmod p$. Let $m:=m_{r,s}$ and $ t_0, t_{1}, \ldots, t_m\in\mathbb{N}$  be as  above. Then the order of $a^{ip^r}b^{jp^s}$ in both $G$ and $G_{m+3}$ is $p^{t(r,s)}$ where
\[
t(r,s)=(m+2)n-(t_0+\cdots+t_{m-1}+s(m+1)+r+R_0),
\]
and if $p^{n-s}\mid S[r]$, that is, when $m=0$, we take $t_0+\cdots +t_{m-1}=0$.
\end{lemma}

\begin{proof}
Clearly the order of $a^{ip^r}b^{jp^s}$ is a multiple of $p^{n-r}$. We have  
\begin{equation}\label{eq order first step}
\begin{split}
&\psi\bigg((a^{ip^r}b^{jp^s})^{p^{n-r}}\bigg)=\psi\bigg((b^{jp^s})^{a^{(p^{n-r}-1)ip^r}}\cdots (b^{jp^s})^{a^{2ip^r}}(b^{jp^s})^{a^{ip^r}}b^{jp^s}\bigg)\\
&=\bigg((a^{jp^{s+R_0}})^*,\overset{p^{r}-1}\ldots,(a^{jp^{s+R_0}})^*, a^{jp^sS[r]}(b^{jp^s})^{a^{jp^s(e_{ip^r}+e_{2ip^r}+\cdots +e_{p^r})}},\\
&\qquad\! (a^{jp^{s+R_0}})^*,\overset{p^{r}-1}\ldots,(a^{jp^{s+R_0}})^*, a^{jp^sS[r]}(b^{jp^s})^{a^{jp^s(e_{ip^r}+e_{2ip^r}+\cdots +e_{2p^r})}},\,\ldots\,,\,\ldots\,,\\
&\qquad 
(a^{jp^{s+R_0}})^*,\overset{p^{r}-1}\ldots,(a^{jp^{s+R_0}})^*,a^{jp^sS[r]}(b^{jp^s})^{a^{jp^s(e_{ip^r}+e_{2ip^r}+\cdots+e_{(p^{n-r}-1)p^r})}},\\
&\qquad (a^{jp^{s+R_0}})^*,\overset{p^{r}-1}\ldots,(a^{jp^{s+R_0}})^*,a^{jp^sS[r]}b^{jp^s}\bigg)
\end{split}
\end{equation}
where  the $*$ denotes an  unspecified exponent, and the last equality follows from the fact that the set $\{ip^r,2ip^r,\ldots,(p^{n-r}-1)ip^r,p^n\}$ coincides, modulo $p^n$, with $\{p^r,2p^r,\ldots,(p^{n-r}-1)p^r,p^n\}$. So the components of $\psi((a^{ip^r}b^{jp^s})^{p^{n-r}})$ in positions a multiple of $p^r$ are conjugates of $a^{jp^sS[r]}b^{jp^s}$ and the other components are powers of $a^{jp^{s+R_0}}$ whose order is at most $p^{n-s-R_0}$. 

Suppose that $p^{n-s}\mid S[r]$. Then using (\ref{eq order first step}), we have 
\[
\begin{split}
&\psi\bigg((a^{ip^r}b^{jp^s})^{p^{n-r}}\bigg)=\bigg((a^{jp^{s+R_0}})^*,\overset{p^{r}-1}\ldots,(a^{jp^{s+R_0}})^*, (b^{jp^s})^{a^{jp^s(e_{ip^r}+e_{2ip^r}+\cdots +e_{p^r})}}, \,\ldots\, ,\,\\
& \qquad\qquad\qquad\qquad\qquad\qquad\qquad\qquad\qquad\qquad\qquad (a^{jp^{s+R_0}})^*,\overset{p^{r}-1}\ldots,(a^{jp^{s+R_0}})^*,b^{jp^s}\bigg).
\end{split}
\]
Note that the order of $b^{jp^s}$ is $p^{n-s-R_0}$. Thus, the order of $a^{ip^r}b^{jp^s}$ in~$G$ is  $p^{t(r,s)}$, for
\[
t(r,s)=2n-(r+s+R_0).
\]
Also since $(b^{jp^s})^{p^{n-s-R_0-1}}\not\in \st_G(2)$, it follows that $(a^{ip^r}b^{jp^s})^{p^{t(r,s)-1}}\not\in \st_G(3)$. Thus the element $a^{ip^r}b^{jp^s}$ is of order $p^{t(r,s)}$ in~$G_3$ as well.

It remains to settle the case $p^{n-s}\nmid S[r]$, which we will now assume till the end of the proof. Since by hypothesis $p^{t_0+s+1}\nmid jp^s{S[r]}$, the order of $a^{jp^sS[r]}b^{jp^s}$ is  a multiple of $p^{n-(t_0+s)}$. Applying $\psi$ we get
\begin{align*}
&\psi_2^{p^n}\bigg((a^{ip^r}b^{jp^s})^{p^{n-r}p^{n-(t_0+s)}}\bigg)=\psi\bigg((a^{jp^sS[r]}b^{jp^s})^{p^{n-(t_0+s)}}\bigg)\\
&=\psi\bigg((b^{jp^s})^{a^{(p^{n-(t_0+s)}-1)jp^sS[r]}}\cdots (b^{jp^s})^{a^{2jp^sS[r]}}(b^{jp^s})^{a^{jp^sS[r]}}b^{jp^s}\bigg).
\end{align*}
As before, 
it follows that 
$\{jp^sS[r],2jp^sS[r],\ldots,(p^{n-(t_0+s)}-1)jp^sS[r],p^n\}$
coincides, modulo~$p^n$, with the set $\{p^{t_0+s},2p^{t_0+s},\ldots,p^n-p^{t_0+s
},p^n\}$. This means that the $b$'s appear in all positions a multiple of $p^{t_0+s}$ and only in these positions, 
and we can write each of the elements in these positions as 
\[
a^{jp^sS[t_0+s]}(b^{jp^s})^{a^{jp^s\big(e_{jp^sS[r]}+ e_{2jp^sS[r]}+\cdots +e_{\ell p^{t_0+s}}\big)}}
\]
for certain $\ell\in\{1,\ldots,p^{n-(t_0+s)}-1\}$.
Since we are interested in the order of these elements, up to conjugation the order of these elements coincides with the order of the element $a^{jp^sS[t_0+s]}b^{jp^s}$. 
In the other components of $\psi((a^{jp^sS[r]}b^{jp^s})^{p^{n-(t_0+s)}})$, there is a power of $a^{jp^{s+R_0}}$ whose order is at most $p^{n-s-R_0}$. By hypothesis 
$p^{t_1+s+1}\nmid p^sS[t_0+s]$ thus the order of $a^{jp^sS[t_0+s]}b^{jp^s}$ is at least $p^{n-(t_1+s)}$. 
Proceeding in this way, after $m-1$ further steps we find
\begin{align*}
&\psi_{m+1}^{p^{mn}}\bigg((a^{ip^r}b^{jp^s})^{p^{n-r}p^{n-(t_0+s)}p^{n-(t_1+s)}\cdots p^{n-(t_{m-1}+s)}}\bigg)
\\
&=\psi\bigg((b^{jp^s})^{a^{(p^{n-(t_{m-1}+s)}-1)jp^sS[t_{m-2}+s]}}\cdots (b^{jp^s})^{a^{2jp^sS[t_{m-2}+s]}} (b^{jp^s})^{a^{jp^s S[t_{m-2}+s]}}(b^{jp^s})\bigg)\\
&= \bigg((a^{jp^{s+R_0}})^*,\overset{p^{t_{m-1}}-1}\ldots,(a^{jp^{s+R_0}})^*,a^{jp^s S[t_{m-1}+s]}(b^{jp^s})^{a^*},(a^{jp^{s+R_0}})^*,\overset{p^{t_{m-1}}-1}\ldots,(a^{jp^{s+R_0}})^*,\\
&\qquad a^{jp^sS[t_{m-1}+s]}(b^{jp^s})^{a^*},\,\ldots\,,\, \ldots\,,
(a^{jp^{s+R_0}})^*,\overset{p^{t_{m-1}}-1}\ldots,(a^{jp^{s+R_0}})^* ,a^{jp^s S[t_{m-1}+s]}(b^{jp^s})\bigg)\\
&= \bigg((a^{jp^{s+R_0}})^*,\overset{p^{t_{m-1}}-1}\ldots,(a^{jp^{s+R_0}})^*,(b^{jp^s})^{a^*},\,\,\ldots\,\,,\,(a^{jp^{s+R_0}})^*,\overset{p^{t_{m-1}}-1}\ldots,(a^{jp^{s+R_0}})^* ,b^{jp^s}\bigg)
\end{align*}
where the last equality holds since $p^{t_{m}+s}\mid p^{s}S[t_{m-1}+s]$, thus $p^n\mid p^s S[t_{m-1}+s]$ by the definition of $t_{m}$.  Since by hypothesis $j\not\equiv 0 \pmod p$ it follows that the order of the elements in positions a multiple of $p^{t_{m-1}}$ is $p^{n-s{-R_0}}$ and the orders of the other elements  are at most $p^{n-s-R_0}$. This proves that the order of $a^{ip^r}b^{jp^s}$ is $p^{t(r,s)}$, where 
\[
t(r,s)=(m+2)n-(t_0+\cdots+t_{m-1}+s(m+1)+r+R_0).
\]

Finally, let us prove that the element $(a^{ip^r}b^{jp^s})^{p^{t(r,s)-1}}\not \in \St_G(m+3)$, which implies that $a^{ip^r}b^{jp^s}$ is of order $p^{t(r,s)}$ in $G_{m+3}$. We observe that 
\begin{align*}
&\psi_{m+1}^{p^{mn}}\bigg((a^{ip^r}b^{jp^s})^{p^{t(r,s)-1}}\bigg)=\bigg((a^{jp^{n-1}})^*,\overset{p^{t_{m-1}}-1}\ldots,(a^{jp^{n-1}})^*,(b^{jp^{n-R_0-1}})^{a^*},\\
&\qquad\qquad\qquad\qquad\qquad\qquad\qquad\qquad \,\ldots\,,\,\ldots\,,(a^{jp^{n-1}})^*,\overset{p^{t_{m-1}}-1}\ldots,(a^{jp^{n-1}})^* ,b^{jp^{n-R_0-1}}\bigg).
\end{align*}
Since the components $(b^{jp^{n-R_0-1}})^{a^*}$ are non-trivial in $G_2$, the result follows.
\end{proof}

In the next three results, we identify a sort of branching structure within certain subgroups of GGS-groups. For convenience, we set $R_{-1}=0$.

\begin{lemma}
\label{lem:fractal}
Let $G=\langle a,b\rangle$  be a  GGS-group acting on the $p^n$-adic tree, for $p$ a  prime and $n\in\mathbb{N}$. Then  for $j\in\{0,1,\ldots,n-1\}$, we have $\varphi_v(\textup{st}_G(v))\ge\langle a^{p^{R_{j-1}}},b\rangle$ for any vertex~$v=u_1\cdots u_j\in A^j$ of length~$j$, where $u_i\equiv 0 \pmod {p^{R_{i-2}}}$ for $i\in\{1,\ldots,j\}$.
\end{lemma}

\begin{proof}
For a vertex~$v$ of length one, the result is clear from considering the sections of $b,b^a,\ldots,b^{a^{p^n-1}}$, which are the generators of $\textup{st}_G(v)$. As the subgroup generated by the sections of~$b$ is $\langle a^{p^{R_0}},b\rangle$, it follows that for a vertex~$w$ of length~$2$, $$b,b^{a^{p^{R_0}}},\ldots,b^{a^{p^n-p^{R_0}}}\in \varphi_v(\textup{st}_G(w)),
$$
where $v$ is the prefix of~$w$ of length~$1$. It then follows that for vertices $w=u_1u_2$, with $u_1,u_2\in A$ such that $u_2\equiv 0\pmod {p^{R_0}}$, we have $\varphi_w(\textup{st}_G(w))\ge\langle a^{p^{R_1}},b\rangle$. The result now follows recursively.
\end{proof}

\begin{lemma}
\label{lem:branching-new}
Let $G=\langle a,b\rangle$  be a  GGS-group acting on the $p^n$-adic tree, for $p$ a  prime and $n\in\mathbb{N}$. Let $\alpha\in\{0,\ldots,n-1\}$ and let $p^R$ be the highest power of $p$ dividing $e_{ip^{\alpha}}$ for all $i\in\{1,\ldots, p^{n-\alpha}-1\}$. Suppose that  
\begin{enumerate}
\item the highest power of $p$ dividing $e_{jp^{\alpha+1}}$ for all $j\in\{1,\ldots, p^{n-\alpha-1}-1\}$ is strictly greater than $p^R$, and 
\item there exists $\ell\in\{1,\ldots, p^{n-\alpha}-1\}$ such that $p^{R+1}\mid e_{\ell p^{\alpha}}$.
\end{enumerate}
Then, writing $N_{\alpha}=\langle a^{p^{\alpha}},b\rangle$ and $N_R=\langle a^{p^{R}},b\rangle$, we have
\[
1 \times \overset{p^{\alpha}-1}\cdots \times  1\times \gamma_3(N_R) \times \cdots \, \cdots \times 1 \times \overset{p^{\alpha}-1}\cdots \times  1\times \gamma_3(N_R) = \psi\big(\gamma_3(\st_{N_{\alpha}}(1))\big).
\]
\end{lemma}

\medskip

Note that if $\alpha=n-1$, condition (i) is vacuous.

\begin{proof}
By assumption (i), there exists $k\in\{1,\ldots,p^{n-\alpha}-1\}$ with $k\not\equiv 0 \pmod p$ such that $e_{kp^\alpha}=\kappa p^{R}$ for some $\kappa\not\equiv 0\pmod p$. Up to  replacing $b$ with a suitable power of itself, we may assume without loss of generality that $\kappa=1$.
It suffices to show, since we can conjugate by~$a^{p^{\alpha}}$, that 
\[
1\times\overset{kp^{\alpha}-1}\cdots \times 1\times \gamma_3(N_R)\times 1\times \cdots \times 1\le\psi\big(\gamma_3(\st_{N_{\alpha}}(1))\big).
\]

\underline{Case 1:} Suppose $e_{(p^{n-\alpha}-k)p^{\alpha}}=0$. Then, as $\gamma_3(N_R)=\langle [a^{p^R},b,a^{ p^R}], [a^{ p^R},b,b]\rangle^{N_R}$, 
the result follows from
\begin{align*}
\psi\Big([b,b^{a^{kp^{\alpha}}},b]\Big)&=(1,\overset{kp^{\alpha}-1}\ldots,1,[a^{p^R},b,a^{ p^R}],1,\ldots,1)\\
\psi\Big([b,b^{a^{kp^{\alpha}}}, b^{a^{kp^{\alpha}}}]\Big)&=(1,\overset{kp^{\alpha}-1}\ldots,1, [a^{ p^R},b,b],1,\ldots,1)
\end{align*}
and using Lemma~\ref{lem:fractal}.

\underline{Case 2:} Suppose $e_{(p^{n-\alpha}-k)p^{\alpha}}\ne 0$ and let $\chi\in\{1,\ldots,p^{n-R}-1\}$ be such that $e_{(p^{n-\alpha}-k)p^{\alpha}}=\chi p^R$. 

If $p\mid \chi$, consider
\begin{align*}\psi\Big([b,b^{a^{kp^{\alpha}}},b]\Big)&=(1,\overset{kp^{\alpha}-1}\ldots,1,[a^{ p^R},b,a^{ p^R}],1,\ldots,1, [b,a^{\chi p^R},b]),\\
\psi\Big([b,b^{a^{kp^{\alpha}}},b^{a^{kp^{\alpha}}}]\Big)&=(1,\overset{kp^{\alpha}-1}\ldots,1,[a^{ p^R},b,b],1,\ldots,1,[b,a^{\chi p^R},a^{\chi p^R}]).
\end{align*}
Let us refer to those commutators in the final component as the \emph{error terms}. We proceed to cancel those error terms by respectively multiplying with the following:
\begin{align*}
&\psi\Big(\big([b^{a^{kp^{\alpha}}},b^{\chi},b^{a^{kp^{\alpha}}}]^{-1}\big)^{a^{-kp^\alpha}}\Big)=
(1,\ldots,1, [a^{\chi p^R},b^{\chi},a^{\chi p^R}]^{-1},1,\overset{kp^{\alpha}-1}\ldots,1,[b,a^{\chi p^R},b]^{-1}),\\
&\psi\Big(\big([b^{a^{kp^{\alpha}}},b^{\chi},b^{\chi}]^{-1}\big)^{a^{-kp^\alpha}}\Big)=
(1,\ldots,1,[a^{\chi p^R},b^{\chi},b^{\chi}]^{-1},1,\overset{kp^{\alpha}-1}\ldots,1,[b,a^{\chi p^R},a^{\chi p^R}]^{-1}).
\end{align*}
This introduces a new set of error terms, this time in the $(p^n-kp^\alpha)$th component. We notice that the new set of error terms  involve a higher $p$-power of $a$ or $b$. Hence, after repeating this process a finite number of times, we will eventually have a set of error terms consisting of  commutators with at least one component being the trivial element $a^{p^n}=b^{p^{n-R_0}}$. In other words, we have completely cancelled any error terms, showing that
\begin{align*}
(1,\overset{kp^{\alpha}-1}\ldots,1,[a^{ p^R},b,a^{p^R}],1,\ldots,1)&\in\psi\big(\gamma_3(\st_{N_\alpha}(1))\big),\\
(1,\overset{kp^{\alpha}-1}\ldots,1,[a^{ p^R},b,b],1,\ldots,1)&\in\psi\big(\gamma_3(\st_{N_\alpha}(1))\big).
\end{align*}

Now suppose that $p\nmid \chi$. 
By hypothesis, there exists  $\ell\in\{1,\ldots,p^{n-\alpha}-1\}$ such that $e_{\ell p^{\alpha}}=\lambda p^R$ with $p\mid \lambda $. We claim that we may choose $\ell$ such that $e_{(\ell-k)p^{\alpha}}=\xi p^R$ with $p\nmid \xi$. Indeed, suppose on the contrary that for all choices of~$\ell$ satisfying condition (ii), we have $p^{R+1}\mid e_{(\ell-k)p^\alpha}$. Since $k\not\equiv 0\pmod p$, repeatedly replacing $\ell$ with $\ell-k$, this means that $e_{jp^\alpha}\equiv 0 \pmod {p^{R+1}}$ for all $j\in\{1,\ldots,p^{n-\alpha}-1\}$, a contradiction. Hence the claim. For convenience, we write  $e_{(k-\ell)p^{\alpha}}=\zeta p^R$, and let $\nu\in\{1,\ldots, p^{n-R}-1\}$ be such that $\nu\xi\equiv \chi \pmod{p^{n-R}}$.  Then we have 
\begin{align*}
\psi\Big([b,(b^\nu)^{a^{(k-\ell)p^\alpha}},b ]^{-1}\Big)&=(1,\overset{(k-\ell) p^{\alpha}-1}\ldots,1,[a^{\zeta p^R},b^{\nu},a^{\zeta p^R}]^{-1},1,\ldots,1, [b,a^{\chi p^R},b]^{-1}),\\
\psi\Big([b,b^{a^{kp^{\alpha}}},(b^\nu)^{a^{(k-\ell)p^\alpha}} ]^{-1}\Big)&=(1,\overset{kp^{\alpha}-1}\ldots,1,[a^{p^R},b,a^{\nu\lambda p^R}]^{-1},1,\ldots,1, [b,a^{\chi p^R},a^{\chi p^R}]^{-1}).
\end{align*}
If $p\mid \zeta$, we proceed to cancel the error terms as in the above argument. If $p\nmid \zeta$, let $\theta\in\{1,\ldots,p^{n-R}-1\}$ be such that $\theta\chi\equiv \zeta \pmod {p^{n-R}}$. Then, since $p\mid \lambda$, we can use the following element to proceed:
\[
\psi\Big([(b^\theta)^{a^{kp^{\alpha}}},b^\nu,(b^{\theta\nu})^{a^{(k-\ell)p^\alpha}} ]\Big)=
(1,\overset{kp^{\alpha}-1}\ldots,1,[b^\theta,a^{\nu p^R},a^{\theta\nu\lambda p^R}],1,\ldots,1, [a^{\zeta p^R},b^\nu,a^{\zeta p^R}]).\qedhere
\]
\end{proof}

\begin{lemma}\label{lem:branching}
Let $G=\langle a,b\rangle$  be a GGS-group acting on the $p^n$-adic tree, for $p$ an odd prime and $n\in\mathbb{N}$. Let $R\le n-1$ be 
such that $e_{k p^{R}}\equiv 0\pmod {p^R}$ for all $k\in\{1,\ldots,p^{n-R}-1\}$ and suppose that $e_{k p^{R}}\not\equiv 0\pmod {p^{R+1}}$ for all $k\in\{1,\ldots,p^{n-R}-1\}$. 
Suppose 
further that $S[R]\equiv 0 \pmod {p^{R+1}}$.
Then, writing $N=\langle a^{p^{R}},b\rangle$, we have
\[
1 \times \overset{p^{R}-1}\cdots \times  1\times \gamma_3(N) \times \cdots \, \cdots \times 1 \times \overset{p^{R}-1}\cdots \times  1\times \gamma_3(N)=\psi\big(\gamma_3(\st_N(1))\big).
\]
\end{lemma}

\begin{proof}
Let $e_{kp^{R}}=i_k p^{R}$ for $k,i_k\in\{1,\ldots,p^{n-R}-1\}$ with $i_k\not\equiv 0\pmod p$. 
Further, by replacing $b$ with an appropriate power, we may assume that $i_1=1$. As reasoned in the previous proof, it suffices to show  
that 
\[
1\times\overset{p^{R}-1}\cdots \times 1\times \gamma_3(N)\times 1\times \cdots \times 1\le\psi\big(\gamma_3(\st_N(1))\big).
\]
We proceed as in~\cite[Lem.~3.2]{FAZR2} by 
 considering the following two cases: 
\begin{enumerate}
    \item[(a)] There exists an $\ell \in\{2,\ldots,p^{n-R}-2\}$ such that $i_{\ell}^2-i_{\ell-1}i_{\ell+1}\not \equiv 0\pmod p$.
 \item[(b)] For all $\ell\in\{2,\ldots,p^{n-R}-2\}$, we have $i_{\ell}^2-i_{\ell-1}i_{\ell+1} \equiv 0\pmod p$.
\end{enumerate}
Note that if $p=3$ and $R=n-1$, then Case (b) trivially holds.

Suppose we are in Case (a). Let us define $g_{\ell}$ as follows:
\[
g_{\ell}=\big((b^{i_{\ell}})^{a^{p^{R}}}b^{-i_{\ell-1}}\big)^{a^{-\ell p^{R}}}
\]
Writing $\mu = i_{\ell}^2-i_{\ell-1}i_{\ell+1}$, we have
\[
\psi(g_{\ell})=(*, \overset{p^{R}-1}\ldots, *,a^{\mu p^{R}},*,\ldots,*,1)
\]
where $*$ represents unspecified elements.
By assumption, we have $\mu\not \equiv 0\pmod p$. Hence there is a power~$g$ of~$g_\ell$ satisfying
\[
\psi(g)= (*, \overset{p^{R}-1}\ldots, *, a^{p^{R}}, *,\ldots, *,1).  
\]
Furthermore,  for simplicity let $\lambda:=i_{p^{n-R}-1} \not\equiv 0\pmod p$. Then
\[
\psi\Big(b^{a^{p^{R}}} (b^{a^{-p^{R}}})^{-\lambda}\Big)=(*, \overset{p^{R}-1}\ldots, *, ba^{-\lambda e_{2p^{R}}}, *,\ldots, *,1),  
\]
and multiplying by an appropriate power of~$g$ we obtain an element $h\in \st_N(1)$ with
\[
\psi(h)=(*, \overset{p^{R}-1}\ldots, *, b, *,\ldots, *,1).
\]
Then the result follows from:
\begin{align*}
\psi\Big([b,b^{a^{p^{R}}},g]\Big)&=(1,\overset{p^{R}-1}\ldots,1,[a^{p^{R}},b,a^{p^{R}}],1,\ldots,1)\\
\psi\Big([b,b^{a^{p^{R}}},h]\Big)&=(1,\overset{p^{R}-1}\ldots,1,[a^{p^{R}},b,b],1,\ldots,1)
\end{align*}

We now  suppose that we are in Case~(b). It then follows, for $j\in\{2,\ldots,p^{n-R}-1\}$, that $i_{j}=\lambda^{j-1}$ for some $\lambda{\not\equiv 0\pmod p}$. However, as $S[R]\equiv 0 \pmod {p^{R+1}}$, we  have $\lambda\not\equiv 1\pmod p$. Further we have $\lambda^{p^{n-R}-1}p^{R}\equiv p^{R} \pmod {p^{R+1}}$.

Note that
\begin{align*}
    &\psi\Big( b (b^{a^{p^{R}}})^{-\lambda}\Big)=\\
 &\qquad\qquad(*, \overset{p^{R}-1}\ldots, *, a^{p^{R}}b^{-\lambda},*, \overset{p^{R}-1}\ldots, *, 1, \,\ldots \, \ldots, *, \overset{p^{R}-1}\ldots, *, 1,*, \overset{p^{R}-1}\ldots, *, ba^{-p^{R}+{\nu p^{R+1}}} )
   \end{align*}
for some $\nu$,    and 
$*$ represents some power of~$a$. Let $\chi$ be such that $\lambda\chi\equiv 1\pmod {p^{n-R_0}}$; recall that $p^{n-R_0}$ is the order of~$b$. Then
  \begin{align*}
   & \psi\Big( b^\chi\big(b^{-\chi\lambda}\big)^{a^{p^R}}\Big)=\psi\Big( b^\chi \big(b^{-1}\big)^{a^{p^R}}\Big)= \\
 &\qquad\quad(*, \overset{p^{R}-1}\ldots, *,a^{\chi p^{R}}b^{-1}, *, \overset{p^{R}-1}\ldots, *, 1,\,\ldots \, \ldots, *, \overset{p^{R}-1}\ldots, *, 1,*, \overset{p^{R}-1}\ldots, *, b^\chi *  ).
\end{align*}
Thus, since $\chi\not\equiv 1\pmod p$ and hence $$
N=\langle a^{\chi p^{R}}b^{-1}, ba^{-p^{R}+\nu p^{R+1}}   \rangle=\langle a^{(\chi-1)p^R+\nu p^{R+1}} , ba^{-p^{R}+\nu p^{R+1}}   \rangle,$$
the result follows from the observations below:
\begin{align*}
    \psi\Big( \big[b^{(1-\nu p)},b^{a^{p^{R}}} ,b^{a^{p^{R}}} (b^{a^{2p^{R}}})^{-\lambda}\big] \Big)&=(1, \overset{p^{R}-1}\ldots, 1, [a^{p^{R}-\nu p^{R+1}} , b, ba^{-p^{R}+\nu p^{R+1}}  ],1,\ldots,1),\\
     \psi\Big( \big[ (b^{a^{2p^{R}}})^{\lambda}   ,b^{a^{p^{R}}} ,b^\chi \big(b^{-\chi\lambda}\big)^{a^{p^{R}}}\big] \Big)&=(1, \overset{p^{R}-1}\ldots, 1, [a^{p^{R}-\nu p^{R+1}} , b, a^{\chi p^{R}}b^{-1} ],1,\ldots,1). \qedhere
\end{align*}
\end{proof}

Next, we record some properties of certain finite GGS-groups 
acting on the $2^n$-adic tree, which will be needed in Section~\ref{sec:exceptional}.

\begin{lemma}
\label{lem:stab-trivial}
Let $G$ be a periodic GGS-group acting on the $2^n$-adic  tree  for $n\ge 2$, and assume that $R_0=n-1$. Then $\St_G(4)$ is trivial. Furthermore $\St_G(3)$ is non-trivial if and only if $e_i = e_{2^{n-1} + i}$ for all $i\in\{1,\ldots, 2^{n-1} - 1\}$.
\end{lemma}

\begin{proof}
Note that as $R_0=n-1$, we have that $2^{n-1}$ divides each $e_i$. Thus the $e_i$'s are $0$ or~$2^{n-1}$. Since $G$ is periodic, it follows that $e_{2^{n-1}}=0$ and hence  $\big\langle b, b^{a^{2^{n-1}}} \big\rangle\cong C_2\times C_2$. 

Let $g$ be an element in $\St_G(4)$. Certainly $\varphi_{\omega}(g)\in \St_G(1)$ for all $\omega\in A$. Since 
\[
\psi( \St_G(1)) \cap \St_G(1) \times \cdots \times \St_G(1)\subseteq \big\langle b, b^{a^{2^{n-1}}} \big\rangle\times \overset{2^n}\cdots\times \big\langle b, b^{a^{2^{n-1}}} \big\rangle,
\]
it follows that 
\[
\varphi_{\omega}(g)\in\big\{1, b, b^{a^{2^{n-1}}},  bb^{a^{2^{n-1}}}\big\}.
\]
As $\big\langle b, b^{a^{2^{n-1}}} \big\rangle \cap\St_G(3)=1$, the first claim follows.

For the second statement, let $g$ be a non-trivial element in $\st_G(3)$, in particular $\varphi_{\omega}(g)\in\st_G(2)$ for all $\omega\in A$. From the previous paragraph, it follows that $\varphi_{\omega}(g)\in\big\{1, bb^{a^{2^{n-1}}}\big\}$. Thus there exists $\omega\in A$ such that $\varphi_{\omega}(g)=bb^{a^{2^{n-1}}}$. We observe that
\[
\psi\Big(bb^{a^{2^{n-1}}}\Big) = (a^{e_1 + e_{2^{n-1} +1}}, \ldots, a^{e_{2^{n-1}-1} + e_{2^n -1}} , b, a^{e_1 + e_{2^{n-1} +1}}, \ldots ,a^{e_{2^{n-1}-1} + e_{2^n -1}}, b),
\]
so  $\varphi_{\omega}(g)\in\St_G(2)$  if and only if $e_i = e_{2^{n-1} + i}$ for all $i\in\{1,\ldots, 2^{n-1} - 1\}$, as required.
\end{proof}

\begin{lemma}\label{lem:nilpotent}
Let $G$ be a periodic GGS-group acting on the $2^n$-adic  tree  for $n\ge 2$.  Suppose $R_0=n-1$ and $S[0]\equiv 0\pmod {p^n}$. 
Then $\gamma_{2^n+1}(G)\le \st_G(1)'$. Furthermore $[b,a,\overset{2^{n-1}}\ldots,a]^2=1$ and $[b,a,\overset{2^{n-1}+i}\ldots,a]^2\in \gamma_{2^n+i+2}(G)$  for all $i\in\mathbb{N}$.
\end{lemma}

\begin{proof}
As $\st_G(1)/\st_G(1)'$ is elementary abelian, we have  $(G')^2\le \st_G(1)'$.
Therefore,  from~\cite[Prop.~1.1.32(ii)]{McKay}, we have 
\[
1=[a^{2^n},b]\equiv[b,a,\overset{2^n}\ldots,a] \pmod {\st_G(1)'},
\]
which gives $\gamma_{2^n+1}(G)\le \st_G(1)'$, as required. 

Also from~\cite[Prop.~1.1.32(ii)]{McKay}, we obtain
\[
[a^{2^{n-1}},b]\equiv [a,b,a,\overset{2^{n-1}-1}\ldots,a] \pmod {\st_G(1)'}
\]
and so $[b,a,\overset{2^{n-1}}\ldots,a]^2=1$, as $\st_G(1)'$ is elementary abelian and $[\st_G(1),\st_G(1)']=1$.

For the last statement, we proceed by induction on $i$. For $i=1$, the statement follows from
\[
1=\big[[b,a,\overset{2^{n-1}}\ldots,a]^2,a\big]=[b,a,\overset{2^{n-1}+1}\ldots,a]^2\big[b,a,\overset{2^{n-1}+1}\ldots,a, [b,a,\overset{2^{n-1}}\ldots,a]\big],
\]
and similarly for the inductive step, as
\[
\big[[b,a,\overset{2^{n-1}+i}\ldots,a]^2,a\big]=[b,a,\overset{2^{n-1}+i+1}\ldots,a]^2\big[b,a,\overset{2^{n-1}+i+1}\ldots,a, [b,a,\overset{2^{n-1}+i}\ldots,a]\big].\qedhere
\]
\end{proof}

\begin{lemma}\label{lem:cyclic}
Let $G$ be a GGS-group satisfying the conditions of  Lemma~\ref{lem:nilpotent}. Then $\gamma_i(G)/\gamma_{i+1}(G)$ is isomorphic to a subgroup of $C_2\times C_2$, for $i\ge 2$.

Furthermore
$|\gamma_i(G):\gamma_{i+1}(G)|\le 2$  for $i>2^n$.
\end{lemma}

\begin{proof}
We notice that
\[
\psi(\st_G(1)') \subseteq S:=\Big\{
\big( [b,a^{2^{n-1}}]^{\epsilon_1},\ldots, [b,a^{2^{n-1}}]^{\epsilon_{2^n}}\big)\mid \epsilon_1,\ldots, \epsilon_{2^n}\in\{0,1\}\Big\}.
\]
Since $|S|=2^{2^n}$, it follows that $\log_2|\st_G(1)'|\le 2^n$.

First we consider the case when the defining vector $\mathbf{e}$ is non-symmetric, that is, there exists $j\in\{1,\ldots,2^{n-1}-1\}$ with $e_j \ne e_{2^n-j}$. Then, there exists some $k\in\{1,\ldots,2^n-1\}$ such that
\[
\psi([b,b^{a^k}])=\big(1,\ldots,1,[b,a^{2^{n-1}}]\big)\in\psi(\st_G(1)').
\]
Hence for $\mathbf{e}$ non-symmetric, we have $\log_2|\st_G(1)'|=2^n$, and
\begin{align*}
    \st_G(1)'&=\langle [b,b^{a^k}],[b,b^{a^k}]^a,[b,b^{a^k}]^{a^2},\ldots, [b,b^{a^k}]^{a^{2^n-1}}\rangle\\
    &=\langle [b,b^{a^k}],[b,b^{a^k},a],[b,b^{a^k},a^2],\ldots, [b,b^{a^k},{a^{2^n-1}}]\rangle\\
    &=\langle [b,b^{a^k}],[b,b^{a^k},a],[b,b^{a^k},a,a],\ldots, [b,b^{a^k},a,\overset{2^n-1}\ldots, a]\rangle.
\end{align*}
Indeed, the last equality follows from the more general fact that
\[
\langle h,[h,a],[h,a^2],\ldots, [h,{a^{\ell}}]\rangle = \langle h,[h,a],[h,a,a],\ldots, [h,a,\overset{\ell}\ldots, a]\rangle
\]
for $h\in\st_G(1)'$ and $0\le \ell \le 2^n-1$. This is proved by induction on~$\ell$, using the identity $[h,a^{\ell+1}]=[h,a][h,a^\ell][h,a^\ell,a]$.

From this, we deduce that for each $i\ge 2$, either  $\st_G(1)'\cap\gamma_i(G)=\st_G(1)'\cap\gamma_{i+1}(G)$ or $(\st_G(1)'\cap\gamma_i(G))/(\st_G(1)'\cap\gamma_{i+1}(G))\cong C_2$. As
$$
\gamma_i(G)=\langle [b,a,\overset{i-1}\ldots,a]\rangle (\st_G(1)'\cap \gamma_i(G))\gamma_{i+1}(G),
$$
the result follows, with the final statement coming from the fact that $[b,a,\overset{2^n}\ldots,a]\in\st_G(1)'$.

Now we suppose that $\mathbf{e}$ is symmetric, so  $e_j = e_{2^n-j}$ for all $j\in\{1,\ldots, 2^{n-1}-1\}$.   Then there exists a minimal $k\in\{1,\ldots,2^{n-1}-2\}$ such that 
\[
\big(1,\ldots,1,[b,a^{2^{n-1}}],1,\overset{k-1}\ldots,1,[b,a^{2^{n-1}}],1,\ldots,1\big)\in\psi(\st_G(1)').
\]
Let $g\in\st_G(1)'$ be such an element, chosen such that $g\in\gamma_{\ell}(G)$ with  $\ell$ minimal. We observe that $k=2^{\lambda}$ for some $\lambda\in\{0,1,\ldots,n-2\}$, else $k$ will not be minimal.
Also,
\[
\big(1,\ldots,1,[b,a^{2^{n-1}}]\big)\notin\psi(\st_G(1)').
\]
Looking at  the $2^n\times 2^n$ circulant matrix~$C$ defined by $(1,0,\overset{2^{\lambda}-1}\ldots,0,1,0,\ldots,0)\in(\mathbb{F}_2)^{2^n}$, a direct generalisation of \cite[Lem.~2.7(i)]{FAZR2} yields that the rank of~$C$ is $2^n-2^{\lambda}$, as the polynomial associated to the circulant matrix is $1+x^{2^\lambda}=(1+x)^{2^\lambda}$. Therefore we obtain that $\log_2|\st_G(1)'|=2^n-2^{\lambda}$,
and
\begin{align*}
\st_G(1)'&=\langle g,g^a,g^{a^2},\ldots,g^{a^{2^n-2^\lambda-1}}\rangle\\
&=\langle g,[g,a],[g,a,a],\ldots,[g,a,\overset{2^n-2^\lambda-1}\ldots,a]\rangle.
\end{align*}
The result now follows as in the non-symmetric case.
\end{proof}

\begin{lemma}\label{lem:all-conjugate}
Let $G$ be a GGS-group satisfying the conditions of  Lemma~\ref{lem:nilpotent}. Then $\langle [b,a,\overset{2^n -1}\ldots,a,b]\rangle^G =\{[(ab)^{2^n},g]\mid g\in G\}$.
\end{lemma}

\begin{proof}
Setting $h=(ab)^{2^n}$, we know from ~\cite[Prop.~1.1.32(i)]{McKay} that
\[
h= [b,a,\overset{2^{n-1} -1}\ldots,a]^2 [b,a,\overset{2^n -1}\ldots,a]c
\]
for some $c\in\gamma_{2^n}(G)\cap \st_G(1)'$.
Using Lemma~\ref{lem:nilpotent}, we observe that for all $1\leq i \leq 2^n-1$, we have $[h,a^i] \in \gamma_{2^n+1}(G)\leq \St_G(1)'$.
Furthermore, one can check that
\[
[h,a]=[h,b]=[b,a,\overset{2^n -1}\ldots,a,b].
\]
Let $j>2^n$ be maximal such that $[b,a,\overset{2^n -1}\ldots,a,b]\in\gamma_j(G)$.
Thus, since $\log_2|\st_G(1)'|\le 2^n$ as seen in the proof of Lemma~\ref{lem:cyclic}, we have
\[
\gamma_j(G)=
\langle
[h,b], [h,b,a], \ldots, [h,b, a,\overset{2^n -1}\ldots,a]
\rangle.
\]
Our goal is now to show that
\[
\{[h,g]\mid g\in G\}
=
\langle
[h,b], [h,b,a], \ldots, [h,b, a,\overset{2^n -1}\ldots,a]
\rangle.
\]
To this purpose, as seen in the previous proof it suffices to prove
\begin{equation}
\label{set-subgroup equality}
\{[h,g]\mid g\in G\}
=
\langle
[h,b], [h,b]^a, \ldots, [h,b]^{a^{2^n-1}}
\rangle.
\end{equation}

For any $1\leq i\leq 2^n-1$, we have that
\begin{equation}
\label{conjugate by a^i}
[h,b]^{a^i}=[h^{a^i}, b^{a^i}]=[h[h,a^i], b^{a^i}]
=[h, b^{a^i}],
\end{equation}
where the last equality follows from the fact that $[h,a^i] \in \gamma_{2^n+1}(G)\leq \St_G(1)'$.
Next we note that
\begin{equation}
\label{commutator of x and a^k}
[h, a^k]=[h,b][h,b]^a \cdots[h,b]^{a^{k}-1} 
\end{equation}
for all $1\leq k \leq 2^n-1$.
Now, every element $g\in G$ is of the form $g=a^{\lambda} b^{a^{\ell_1}}\cdots b^{a^{\ell_d}}$ for some $d\ge 0$ and $\lambda, \ell_{1}, \ldots, \ell_d\in\{0,1,\ldots, 2^n-1\}$.
Then taking into account (\ref{conjugate by a^i}) and (\ref{commutator of x and a^k}), we get
\[
[h,g]=[h,b]^{a^{\ell_d}}\cdots [h,b]^{a^{\ell_1}}[h,a^{\lambda}],
\]
from which we deduce that (\ref{set-subgroup equality}) holds. 
\end{proof}


\section{Beauville structures for quotients of infinite GGS-groups \\acting on the $p^n$-adic tree}\label{sec:Beauville}

We are now ready to prove that  certain quotients of $G$ are Beauville groups, for $G$ an  infinite periodic GGS-group acting on the $p^n$-adic tree where $p$ is a prime and $n\ge 2$. 
Recall that for a 2-generator finite group~$H$, the sets $\{x_1,x_2\}$ and $\{y_1,y_2\}$ form a Beauville structure for~$H$, if and only if
\begin{equation}\label{intersection1}
\langle x\rangle \cap \langle y^g\rangle =1
\end{equation}
for all $x\in X=\{x_1,x_2,x_1x_2\}$, $y\in Y=\{y_1,y_2,y_1y_2\}$ and $g\in H$. We also remark that we will only consider $G/\textup{St}_G(k)$ for $k$ large enough such that the order of each element in~$X$ (or in $Y$) is the same in $G$ and in $G/\textup{St}_G(k)$. The reason for this will become apparent when we prove Theorem~\ref{thm:main} later in this section.

Recall $m_{r,s}$ from the beginning of Section~\ref{sec:properties}.  For $G$ an infinite GGS-group, recall that  there exists some $1\le q\le n-1$ such that $R_q=R_{q+1}<n$. We write $p^d$ for the order of $[a^{p^{R_q}},b,a^{p^{R_q}}]b$ in~$G$, and we denote by $\lambda'$ the minimal level where all sections of $([a^{p^{R_q}},b,a^{p^{R_q}}]b)^{p^{d-1}}$ that involve $b$ are of the form $(b^*)^{a^*}$ for some unspecified exponents $*$.  For the rest of this section, we set $m=\max\{ m_{0,0}+\lambda', \, m_{1,0}\}$.

\begin{theorem}\label{thm:highest-nonempty}
Let $G$ be an  infinite periodic GGS-group acting on the $p^n$-adic  tree,  for $p$ a  prime and $n\ge 2$.  Then the quotient $G/\textup{St}_G(m+3)$ is a Beauville group.
\end{theorem}

\begin{proof} \underline{Case 1:} Suppose that $e_{kp^{n-1}}\ne 0$ for some $k\in\{1,\ldots,p-1\}$. Observe that by assumption the prime $p$ must be odd. Let $1\le q\le n-1$ be  such that $R:=R_q=R_{q+1}<n$. Suppose first that $e_{\ell p^{n-1}}= 0$ for some $\ell\in\{1,\ldots,p-1\}$.  Then  writing $\mu=m_{0,0}$ and using Lemmata~\ref{lem:branching-new} and \ref{lem:fractal} repeatedly we obtain 
\[
c=\psi_{\mu+1}^{-1}\big((1,\overset{p^{(\mu+1)n}-1}{\ldots},1,[a^{p^{R}},b,a^{p^{R}}])\big)\in \gamma_3(G).
\]

By abuse of notation, we still write $a$ and $b$ for their images in $G_{m+3}$. We set 
\[
X=\{a^{p-1},ab,a^pb\} \quad\text{with}\quad Y=\{  ac,b, acb\}.
\]

Assume first that $x=a^{p-1}$, and consider $y=ac$. We write $p^\tau$ for the order of~$ac$ in~$G_{m+3}$, which is strictly greater than~$p^n$.  In order to prove our claim for this choice of $x$ and $y$, it suffices to show that
\[
\langle a^{p^{n-1}}\rangle \ne \langle (ac)^{p^{\tau-1}}\rangle^g,
\]
for all $g\in G_{m+3}$.
Note that the element $(ac)^{p^{\tau-1}}$ is in $\st_{G_{m+3}}(1)$ and  so is any conjugate of $(ac)^{p^{\tau-1}}$. However $a^{p^{n-1}}\not \in \st_{G_{m+3}}(1)$, hence (\ref{intersection1}) holds here.
 Similarly for $y\in\{b,acb\}$.

 Suppose next that $x=ab$, and consider $y=b$. By Lemma~\ref{orders}, the element~$ab$ has order $p^{t(0,0)}$.
 It suffices to consider the intersection of $\langle (ab)^{p^{t(0,0)-1}}\rangle$ and $\langle b^{p^{n-R_0-1}}\rangle^g$. 
 We observe that $\psi((b^{p^{n-R_0-1}})^g)$ has exactly one component consisting of  purely a non-rooted automorphism in $\st_{G}(1)$, whereas $\psi((ab)^{p^{t(0,0)-1}})$ has $p^n$ such components.
 Therefore (\ref{intersection1}) holds.  For $y=ac$, we observe that $\psi_{\mu+1}((ac)^{p^{\tau-1}})$ has $p^n$ components consisting of non-rooted  automorphisms in~$\gamma_3(G)$, whereas for $\psi_{\mu+1}((ab)^{p^{t(0,0)-1}})$ all components consisting of non-rooted  automorphisms  are in $\st_{G}(1)\backslash \gamma_3(G)$.
 Hence the result holds for this pair.
 
Let $x=a^pb$. Since $(a^pb)^{p^{n-1}}$ has $p^{n-1}$ components consisting of non-rooted automorphisms at the first level, whereas both $(ac)^{p^n}$ and $(acb)^{p^n}$ have $p^n$ such components and $b^{p^{n-R_0-1}}$ has one such component, we have that \eqref{intersection1} holds here.

Hence, it remains to establish~(\ref{intersection1}) for $x=ab$ with $y=acb$. 
Recall from the proof of Lemma~\ref{orders} that 
\begin{align*}
&\psi_{\mu+1}^{p^{\mu n}}\bigg((ab)^{p^{t(0,0)-n+R_0}}\bigg)=\bigg((a^{p^{R_0}})^*,\overset{p^{t_{\mu-1}}-1}\ldots,(a^{p^{R_0}})^*,b^{a^*},\\
&\qquad\qquad\qquad\qquad\qquad\qquad\qquad\qquad\qquad\qquad \,\ldots\,,\,\ldots\,,(a^{p^{R_0}})^*,\overset{p^{t_{\mu-1}}-1}\ldots,(a^{p^{R_0}})^* ,b\bigg).
\end{align*}
However, 
\begin{align*}
&\psi_{\mu+1}^{p^{\mu n}}\bigg((acb)^{p^{t(0,0)-n+R_0}}\bigg)=\bigg((a^{p^{R_0}})^*,\overset{p^{t_{\mu-1}}-1}\ldots,(a^{p^{R_0}})^*,([a^{p^{R}},b,a^{p^{R}}]b)^{a^*},\\
&\qquad\qquad\qquad\qquad\qquad\qquad \,\ldots\,,\,\ldots\,,(a^{p^{R_0}})^*,\overset{p^{t_{\mu-1}}-1}\ldots,(a^{p^{R_0}})^* ,[a^{p^{R}},b,a^{p^{R}}]b\bigg).
\end{align*}
As $\psi([a^{p^{R}},b,a^{p^{R}}]b)$ has more than one section with a non-rooted automorphism, whereas $\psi(b)$ has only one, the result follows. Indeed,
\begin{align*}
\psi([a^{p^{R}},b,a^{p^{R}}]b)&=(a^*,\overset{p^{R}-1}{\ldots},a^*, (ba^{-e_{p^n-p^{R}}}b)^{a^{e_{p^{R}}}}
,a^*,\overset{p^{R}-1}{\ldots},a^*, (a^{e_{p^{R}}}b^{-1}a^{e_{p^{R}}})^{a^{e_{2
p^{R}}}}
,\\
&\qquad 
\qquad\qquad 
\qquad\qquad 
\qquad\qquad 
\qquad a^*,\ldots, a^*,
b^{-1}a^{2e_{p^n-p^{R}}-e_{p^n-2p^{R}}}b
),
\end{align*}
and so if $e_{p^{R}}
=e_{p^n-p^{R}}=0$ then
\[
\psi([a^{p^{R}},b,a^{p^{R}}]b)=(a^*,\overset{p^{R}-1}{\ldots},a^*,  b^2,a^*,\overset{p^{R}-1}{\ldots},a^*, (b^{-1})^{^{e_{2
p^{R}}}},a^*,\ldots, a^*,(a^{-e_{p^n-2p^R}})^b)
\]
and the result is clear. If either $e_{p^{R}}$ or $e_{p^n-p^{R}}$ is non-zero, we see from the proof of Lemma~\ref{orders} that, writing $p^\chi$ for the order of $acb$ in~$G$, the minimal level~$\lambda$ of the tree, for which all $\lambda$-level sections of $(acb)^{p^{\chi-1}}$ that involve $b$ are of the form $(b^*)^{a^*}$ for some unspecified exponents~$*$, is strictly greater than $\mu+1$. In particular, 
all these cases imply that the number of sections of $(acb)^{p^{\chi-1}}$ at level~$\lambda$ consisting of  purely non-rooted automorphisms is  strictly greater than $p^{n-t_{\mu-1}}$, which is the corresponding number for $(ab)^{p^{t(0,0)-1}}$ at level~$\lambda$. 

Suppose now that $e_{\ell p^{n-1}}\ne 0$ for all $\ell\in\{1,\ldots,p-1\}$. The result follows similarly using Lemma~\ref{lem:branching}.

\underline{Case 2:} Suppose that $e_{kp^{n-1}}= 0$ for all $k\in\{1,\ldots,p-1\}$. Since $G$ is infinite,
for some $1\le q\le n-1$ we have $R:=R_q=R_{q+1}<n-1$. Then 
we proceed as in the previous case, however noting that for $p=2$, we replace $c$ with $$\psi_{\mu+1}^{-1}\big((1,\overset{2^{(\mu+1)n}-1}{\ldots},1,[a^{2^{R}},b,a^{2^{R+1}}])\big),
$$
since then,   recalling that $R <n-1$,
\begin{align*}
&\psi([a^{2^{R}},b,a^{2^{R+1}}]b)=\\
&\quad \big(a^*,\overset{2^{R}-1}{\ldots},a^*, (ba^{-e_{ 2^n-2^{R+1}}+e_{ 2^n-2^{R}}})^{a^{e_{2^{R}}}}, a^*,\overset{2^{R}-1}{\ldots}, a^*, (a^{e_{2^R}-e_{2^n-2^{R}}}b)^{a^{e_{2^{R+1}}}},\\
&\,\,\,\quad a^*,\overset{2^{R}-1}{\ldots},a^*, (a^{e_{2^{R+1}}}b^{-1}a^{e_{2^{R}}})^{a^{e_{3\cdot 2^{R}}}}, a^*,\ldots,\ldots, a^*,(a^{e_{2^n-2^R}-e_{2^n-3\cdot 2^R}+e_{2^n-2^{R+1}}})^b\big).\qedhere
\end{align*}
\end{proof}

\begin{proof}[Proof of Theorem~\ref{thm:main}] 
Let $G$ be an infinite periodic GGS-group acting on the $p^n$-adic tree with $p$ a prime and $n\ge 2$. Let $t_0$ be defined according to $r=s=0$,
and set $m_G=m+3$.
By Theorem~\ref{thm:highest-nonempty}, the quotient $G/\textup{St}_G(m_G)$ is a Beauville group with triples $X=\{a^{p-1},ab,a^pb\}$ and $Y=\{  ac,b, acb\}$,  for $c$ defined as in the proof of Theorem~\ref{thm:highest-nonempty}.

 As  the orders of each of the elements in~$X$ 
 are the same in $G/\St_G(k)$ and $G/\St_G(m_G)$ for all $k\ge m_G$, it follows from \cite[Lem.~4.2]{FJ} that $G/\St_G(k)$ is a Beauville group for every $k\ge m_G$.
\end{proof}

\begin{remark}\label{rmk:mixed}
We recall that a \emph{mixed Beauville structure} for a finite group~$H$ is a triple $(H,M,X)$ with $M\le H$ a subgroup of index~2 and $X=\{x,y\}$ satisfying the following properties:
\begin{enumerate}
\item[$\bullet$] $M=\langle x,y\rangle$,
\item[$\bullet$] there exists $h_0\in H\backslash M$ such that $$h_0^{-1}\big(\Sigma(x,y)\big)h_0\cap \Sigma(x,y)=\{1\},$$
\item[$\bullet$] for all $h\in H\backslash M$, we have $h^2\notin \Sigma(x,y)$.
\end{enumerate}
Since $M$ is of index 2, for finite $p$-groups, only 2-groups can admit mixed Beauville structures.  Furthermore, if $H$ admits a mixed Beauville structure $(H,M,X)$, then $M$ must be 2-generated. 
 However, none of the finite 2-groups that we consider are mixed Beauville groups. This is because for $G$  an infinite periodic GGS-group acting on the $2^n$-adic tree for $n\ge 2$, no maximal subgroup of $G/\st_G(k)$ is 2-generated.  Indeed, consider $G_k$ for a fixed $k\ge m_G$. Since $\Phi(G_k)=G_k^2$, there are three maximal subgroups of~$G_k$, which are
    \[
    M_1:=\langle a\rangle G_k^2,\quad  M_2:=\langle b\rangle G_k^2,\quad  M_3:=\langle ab\rangle G_k^2.
    \]
   The Schreier index formula shows that $M_1$, $M_2$ and $M_3$ can be generated by at most 3 elements. Using the Reidemeister-Schreier method,  we obtain 
    \[
    M_1=\langle a,(ab)^2,b^2\rangle, \quad  M_2=\langle a^2, b,aba\rangle=\langle a^2, b,(ab)^2\rangle, \quad  M_3=\langle a^2, ba,ab\rangle=\langle a^2, b^2,ab\rangle. \]
     Clearly,
    \[
    M_3/G_k'=\langle a^2, b^2,ab\rangle/G_k'\cong \langle a^2\rangle \times \langle b^2\rangle \times \langle ab\rangle , \]
    and thus $d(M_3)=3$.
 Suppose for a contradiction that $d(M_1)=2$. 
    Since 
    \[
    M_1/G_k'=\langle a,b^2\rangle 
    G_k'/G_k'\cong \langle a\rangle \times \langle b^2\rangle, 
    \]
     and  $[a,b]\notin \langle a,b^2\rangle$, it follows that $d(M_1)=3$. A similar argument holds for $M_2$.
\end{remark}

\begin{remark}\label{rem:non-periodic}
In the above theorems, we only require that $S[i]\equiv 0 \pmod {p^{i+1}}$ for $i\in I$, for some subset $I\subseteq \{0,1,\ldots,n-1\}$. Therefore there are non-periodic GGS-groups acting on the $p^n$-adic tree, for $n\ge 2$, with quotients being Beauville groups. For instance, if 
\begin{enumerate}
    \item [$\bullet$] $t_0 >2$ for the pairs $r=s=0$  and $r=1, s=0$, and
    \item [$\bullet$]  $R>2$, and 
     \item [$\bullet$] $S[2]\not\equiv 0 \pmod {p^{3}}$, but $S[i]\equiv 0 \pmod {p^{i+1}}$ for all other $i$,
\end{enumerate}  
 the proofs above still hold.
\end{remark}

\begin{remark}\label{rmk:finite}
In some cases, the proofs above can be extended to finite GGS-groups $G$. Since $G$ is finite, there exists 
$d\in\{0,1,\ldots,n-1\}$
such that $R_{d+1}=n$. Writing $\mu=m_{0,0}$, we observe that $d\ge \mu-1$. 
If $d> \mu$, then by repeated application of Lemmata~\ref{lem:branching-new} and \ref{lem:fractal},  
we have an element  $$c:=\psi_{\mu+1}^{-1}\big((1,\overset{p^{(\mu+1)n}-1}{\ldots},1,[a^{p^{R_{\mu}}},b,a^{p^{R_{\mu}}}])\big)\in \gamma_3(G).$$
Setting $X=\{a^{p-1},ab,a^pb\}$ and $Y=\{  ac,b, acb\}$, we are done as in Case 1 of the  proof of Theorem~\ref{thm:highest-nonempty}. 

If $d= \mu$, 
then
using the same $c$ and the sets $X$ and $Y$ as before, 
since $e_{\ell p^{R_\mu}}=0$ for any~$\ell$, we observe that
\[
\psi([a^{p^{R_\mu}},b,a^{p^{R_\mu}}]b)=(a^*,\overset{p^{R_\mu}-1}{\ldots},a^*, b^2
,a^*,\overset{p^{R_\mu}-1}{\ldots},a^*, b^{-1}
,a^*,\ldots, a^*,1
),
\]
and hence if $p$ is odd, the same argument follows.  For $p=2$, the same argument can clearly be used upon replacing $c$ with $$\psi_{\mu+1}^{-1}\big((1,\overset{2^{(\mu+1)n}-1}{\ldots},1,[a^{2^{R_\mu}},b,a^{2^{R_\mu+1}}])\big),
$$
provided $R_\mu <n-1$, since then 
\[
\psi([a^{2^{R_\mu}},b,a^{2^{R_\mu+1}}]b)=(a^*,\overset{2^{R_\mu}-1}{\ldots},a^*, b,a^*,\overset{2^{R_\mu}-1}{\ldots},a^*, b,a^*,\overset{2^{R_\mu}-1}{\ldots},a^*, b^{-1},a^*,\ldots, a^*,1).
\]

However, for  each of these finite groups $G$, one would only obtain finitely many (Beauville) quotients $G/\text{St}_G(k)$, for $k\ge m_G$.

In the next section, we identify a class of finite GGS-groups where no quotient by a level stabiliser is Beauville.
\end{remark}

\section{Finite GGS-groups acting on the $p^n$-adic tree
}\label{sec:exceptional}

We define the  subclass $\mathcal{E}$ of finite GGS-groups  to consist of:
\begin{enumerate}
\item[(i)] the GGS-groups acting on the $p^n$-adic tree, for $p$ any prime and $n\ge 2$, with $t_0=R_0$, $\ldots$ , $t_{m_{0,0}-1}=R_{m_{0,0}-1}$ and $R_{m_{0,0}}=n$;
and 
\item[(ii)] the GGS-groups acting on the $2^n$-adic tree with 
$R_0=n-1$, $S[0]\equiv 0\pmod {p^n}$, and 
 $e_{2^{n-1}}= 0$.
\end{enumerate}

We show here that quotients of the GGS-groups $G\in\mathcal{E}$ by  level stabilisers do not yield Beauville groups. We first consider the subfamily (i) of groups in~$\mathcal{E}$.

\begin{theorem}\label{thm:easy-exceptional}
Let $G\in\mathcal{E}$ be a  GGS-group acting on the $p^n$-adic  tree  for $n\ge 2$, and assume that $t_0=R_0$, $\ldots$ , $t_{m_{0,0}-1}=R_{m_{0,0}-1}$ and $R_{m_{0,0}}=n$,
where $t_0, t_1, \ldots, t_{m_{0,0}-1}$ are defined according to $r=s=0$.
Then  the quotient $G/\textup{St}_G(k)$ is not a Beauville group for all $k\in\mathbb{N}$.
\end{theorem}

\begin{proof}
The result is clear for $k=1$ since $G/\st_G(1)$ is cyclic, so we assume $k\ge 2$.  Let $\{x_1,x_2\}$ and $\{y_1,y_2\}$ be two systems of generators for $G$. At least one of $x_1,x_2,x_1x_2$, call it $z_1$, must be in the coset $a^ib^jG'$ for  $i,j\not \equiv 0 \pmod p$. Likewise for $y_1,y_2,y_1y_2$, and call it $z_2$. 
By Lemma~\ref{orders} and its proof, the order of~$ab$ is~$p^{t(0,0)}$ if $k\ge m_{0,0}+3$, and $p^{2n-R_0}$ if $k=2$. For $3\le k\le m_{0,0}+1$, the order of $ab$ is at least
\[
p^np^{n-R_0}p^{n-R_1}\cdots p^{n-R_{k-2}},
\]
and for $k=m_{0,0}+2$, 
the order of $ab$ is at least
\[
p^np^{n-R_0}p^{n-R_1}\cdots p^{n-R_{m_{0,0}-1}}.
\]

Suppose first that $k=2$. For  $i,j\not \equiv 0 \pmod p$ and $w\in G_2'$, as
\[
\phi_2((a^ib^jw)^{p^n})=(a^{jS[0]},\ldots, a^{jS[0]}) 
\]
it follows that $\langle z_1^{p^{2n-R_0-1}}\rangle =\langle z_2^{p^{2n-R_0-1}}\rangle$, which gives the result.

Now let $k\ge m_{0,0}+3$. Here, writing $t=t(0,0)$, we claim that 
\[
\langle (ab)^{p^{t-1}}\rangle = \langle (a^ib^jw)^{p^{t-1}}\rangle\cong C_p
\]
for all $i,j\not \equiv 0 \pmod p$ and $w\in G'$.

To prove the claim, we first note, writing $\psi(w)=(w_1,\ldots,w_{p^n})$, that the product of all the components $w_1,\ldots, w_{p^{n}}$, irrespective of the order, has zero total exponent in~$a$, and in~$b$, and if written as
\[
(b^{\beta_1})^{a^{\alpha_1}}\cdots(b^{\beta_d})^{a^{\alpha_d}}
\]
for some $d\ge 2$ with $\beta_1,\ldots,\beta_d\in \mathbb{Z}/p^{n-R_0}\mathbb{Z}$ and $\alpha_1,\ldots,\alpha_d\in \mathbb{Z}/p^n\mathbb{Z}$, 
then  we have $\alpha_i\equiv 0 \pmod {p^{R_0}}$ for all $i\in \{1,\ldots,d\}$.

Hence, each component of $\psi\big((a^ib^jw)^{p^n}\big)$ is of the form
\[
a^{jS[0]}b^j v_0,
\]
for some $v_0\in\langle a^{p^{R_0}},b\rangle'$. 
We note that, for  $\omega\in A$ with $\omega\not\equiv  0 \pmod {p^{R_0}}$,
\[
\varphi_\omega\bigg((a^{jS[0]}b^j v_0)^{p^{n-R_0}}\bigg)= \varphi_\omega\bigg((a^{S[0]}b)^{jp^{n-R_0}}\bigg)=
a^{j\gamma_\omega}
\]
where
\[
\gamma_\omega=\sum_{\ell=0}^{p^{n-R_0}-1} e_{\omega+\ell p^{R_0}},
\]
 and for  $\omega\in A$ with $\omega\equiv  0 \pmod {p^{R_0}}$,
\[
\varphi_\omega\bigg(\big((a^{jS[0]}b^j v_0)^{p^{n-R_0}}\big)^{a^{f p^{R_0}}}\bigg)
= a^{jS[R_0]}b^jv_1
\]
for $v_1\in\langle a^{p^{R_1}},b\rangle'$ and $f\in\{0,1,\ldots,p^{n-R_0}-1\}$.

Thus, continuing in this manner, we see that it suffices to compare
\[
(a^{jS[R_{\mu-2}]}b^jv_{\mu-1} )^{p^{n-R_{\mu-1}}}\qquad\text{with}\qquad (a^{S[R_{\mu-2}]}b)^{p^{n-R_{\mu-1}}},
\]
where $\mu=m_{0,0}$ and $v_{\mu-1} \in\langle a^{p^{R_{\mu-1}}},b\rangle'$. Recalling also that $R_{\mu}=n$, equivalently the components of $\psi(b)$ in positions $p^{R_{\mu-1}},2p^{R_{\mu-1}},\ldots, p^n-p^{R_{\mu-1}}$ are trivial, we have
\[
\psi\big((a^{jS[R_{\mu-2}]}b^jv_{\mu-1} )^{p^{n-R_{\mu-1}}}\big)=(a^{j\delta_1},\ldots, a^{j\delta_{p^{R_{\mu-1}}-1}},b^j,\,\,\ldots\,\,, a^{j\delta_1},\ldots, a^{j\delta_{p^{R_{\mu-1}}-1}},b^j)
\]
where, for $q\in\{1,\ldots,p^{R_{\mu-1}}-1\}$,
\[
\delta_q=\sum_{\ell=0}^{p^{n-R_{\mu-1}}-1} e_{q+\ell p^{R_{\mu-1}}}.
\]

As 
\[
\psi\big((a^{S[R_{\mu-2}]}b)^{p^{n-R_{\mu-1}}}\big)=(a^{\delta_1},\ldots, a^{\delta_{p^{R_{\mu-1}}-1}},b,\,\,\ldots\,\,, a^{\delta_1},\ldots, a^{\delta_{p^{R_{\mu-1}}-1}},b),
\]
and, from the above computations, we deduce that all other non-trivial labels in the portraits of $(ab)^{jp^{t-1}}$ and $(a^ib^jw)^{p^{t-1}}$ are equal powers of~$a$. Hence,
the claim follows.

Thus, we have $\langle z_1^{p^{t-1}}\rangle = \langle z_2^{p^{t-1}}\rangle$. 

For $3\le k\le \mu+2$, writing $p^{\hat{t}}$ for the order of~$ab$ modulo $\st_G(k)$, the analogous result $\langle z_1^{p^{\hat t-1}}\rangle = \langle z_2^{p^{\hat t-1}}\rangle$ follows similarly from the component-wise description of 
\[
(a^{jS[0]}b^jv_{0} )^{p^{n-R_{0}}}, \,(a^{jS[R_{0}]}b^jv_{1} )^{p^{n-R_{1}}}, \, \ldots\,,\,(a^{jS[R_{k-4}]}b^jv_{k-3} )^{p^{n-R_{k-3}}},
\]
where we recall that $R_{-1}=0$.
\end{proof}

\begin{theorem}\label{thm:R-0_big}
Let $G$ be a periodic GGS-group acting on the $2^n$-adic  tree  for $n\ge 2$. Suppose that   $R_0=n-1$ and $S[0]\equiv 0\pmod {p^n}$.
Then,  the quotient $G/\textup{St}_G(k)$ is not a Beauville group for all $k\in\mathbb{N}$.
\end{theorem}

\begin{proof}
From Lemma~\ref{lem:stab-trivial}, we only need to consider $k\le 4$. 
For $k=1$ the result is clear since Beauville groups are not cyclic. For the remaining values of $k$, let  $X=\{x_1, x_2, x_1x_2\}$ and $Y=\{ y_1, y_2, y_1y_2\}$ be two systems of generators for~$G_k$.

Assume now $k= 2$.
Then there exists some $z_1\in X$ with $z_1 \equiv a^i \pmod {G_2'}$ for some odd~$i$, and there exists $z_2\in Y$ with $z_2 \equiv a^j \pmod {G_2'}$ for some odd~$j$. Noting that the order of both $z_1$ and $z_2$ is $2^n$, we now prove that $(a^iw)^{2^{n-1}}$ is conjugate to $(a^jv)^{2^{n-1}}$, for $w,v \in G_2'$.

We observe that  $\St_{G_2}(1)$, and hence $G_2'$, is elementary abelian. Furthermore for each odd $i\in\{1,\ldots,2^n-1\}$ we have 
\begin{equation}\label{eq: G/st(2)'}
G_2'\leq \langle [a^i,b], [a^i,b,a^i],\ldots,[a^i,b,a^i,\overset{2^n-2}{\ldots},a^i]\rangle;
\end{equation}
compare Lemma~\ref{lem:nilpotent} and the proof of Lemma~\ref{lem:cyclic}. It then follows from~\cite[Prop.~1.1.32(i)]{McKay} that
\[
(a^iw)^{2^{n-1}}= a^{i2^{n-1}}[w, a^i , \overset{2^{n-1} -1}{\ldots} , a^i].
\]
From~\eqref{eq: G/st(2)'} there exist $j_1,\ldots,j_{2^{n}-1}\in\{0,1\}$ such that $w$ can be expressed as: 
\[
w= [a^i,b]^{j_1}[a^i,b,a^i]^{j_2}\cdots[a^i,b,a^i,\overset{2^n-2}{\ldots},a^i]^{j_{2^n-1}}
\]

Now for $h\in\St_{G_2}(1)$, bearing in mind that $\St_{G_2}(1)$ is elementary abelian, we obtain from 
\cite[Prop.~1.1.32(ii)]{McKay} that
$[a^{i2^{n-1}},h]=[a^i,h,a^i,\overset{2^{n-1}-1}{\ldots},a^i]$ in~$G_2$.

Thus, setting $h=b^{j_1}[a^i,b]^{j_2}\cdots[a^i,b,a^i,\overset{2^n-3}{\ldots},a^i]^{j_{2^n-1}}$ we deduce that 
\[
[a^{i2^{n-1}},h]=[w, a^i , \overset{2^{n-1} -1}{\ldots} , a^i].
\]
Indeed,
\begin{align*}
&[a^{i2^{n-1}},h]\\
&= [a^i,b^{j_1},a^i,\overset{2^{n-1}-1}{\ldots},a^i][a^i,[a^i,b]^{j_2},a^i,\overset{2^{n-1}-1}{\ldots},a^i]\cdots[a^i,[a^i,b,a^i,\overset{2^n-3}{\ldots},a^i]^{j_{2^n-1}},a^i,\overset{2^{n-1}-1}{\ldots},a^i]\\
&= [[a^i,b]^{j_1},a^i,\overset{2^{n-1}-1}{\ldots},a^i][[a^i,b,a^i]^{j_2},a^i,\overset{2^{n-1}-1}{\ldots},a^i]\cdots[[a^i,b,a^i,\overset{2^n-2}{\ldots},a^i]^{j_{2^n-1}},a^i,\overset{2^{n-1}-1}{\ldots},a^i]\\
&=[w, a^i , \overset{2^{n-1} -1}{\ldots} , a^i]
\end{align*}
where the second equality holds since  $\St_{G_2}(1)'$ is trivial and $[a,x]=[x,a]$ for all $x\in\St_{G_2}(1)$. 

Thus we have 
\[
(a^iw)^{2^{n-1}}=a^{i2^{n-1}}[a^{i2^{n-1}},h]=(a^{i2^{n-1}})^h= (a^{2^{n-1}})^h
\]
where the last equality holds since $a$ has order $2^{n}$ and $i=1+2d$ for some $d\in\mathbb{N}$. This proves that for all $i\in\{1,\ldots,2^n-1\}$ and $w\in G_2'$ the element $(a^iw)^{2^{n-1}}$ is conjugate to~$a^{2^{n-1}}$. Hence $z_1$ and $z_2$ are conjugate and $G_2$ is not a Beauville group.

Now suppose  $\St_G(3)\ne 1$.
Then there exists some $z_1 \in X$ with $z_1 \equiv a^ib \pmod {G_3'}$ for some odd~$i$, and there exists $z_2\in Y$ with $z_2 \equiv a^jb \pmod {G_3'}$ for some odd~$j$. Observe that for $w\in G_3'$, we have 
\[
\phi_3\big((a^ibw)^{2^{n}}\big)=\big(b^{a^{\lambda_{1}\cdot 2^{n-1}}}, \ldots, b^{a^{\lambda_{2^n}\cdot2^{n-1}}} \big)
\]
where  $\lambda_1,\ldots,\lambda_{2^n}\in \{0,1\}$.
Since by Lemma~\ref{lem:stab-trivial} the elements $b$ and $b^{a^{2^{n-1}}}$ coincide modulo $\St_G(2)$, we further have
\[
\phi_3((a^ibw)^{2^{n}})= (b,\ldots,b).
\]
This proves that $z_1^{\,2^{n}}$ and $z_2^{\,2^{n}}$ are equal, and thus $G_3$ is not a Beauville group when $\st_G(3)\ne 1$.

We now consider the remaining cases, where here $G_k=G$.
As before there exists some $z_1\in X$ with $z_1 \equiv a^ib \pmod {G'}$ for some odd~$i$, and there exists $z_2\in Y$ with $z_2 \equiv a^jb \pmod {G'}$ for some odd~$j$. Therefore it suffices to show, for  $\delta\in\{0,1,\ldots, 2^{n-1}-1\}$ and $w\in G'$, that $(a^{1+2\delta}bw)^{2^n}$  is conjugate to $(ab)^{2^n}$ in~$G$.

Recall that $(G')^2\le \st_G(1)'$ and that $\st_G(1)'$ is elementary abelian. Hence, from~\cite[Prop.~1.1.32(i)]{McKay} we have
\[
    (ab)^{2^n}= [b,a,\overset{2^{n-1} -1}\ldots,a]^2 [b,a,\overset{2^n -1}\ldots,a]c
\]
for some $c\in\gamma_{2^n}(G)\cap \st_G(1)'$.
Let $i_1,i_2,i_3\ge 2^n$ be such that
\begin{align*}
  \big[  [b,a,\overset{2^{n-1} -1}\ldots,a]^2,a\big]&\in \gamma_{i_1}(G)\backslash \gamma_{i_1+1}(G)\\
  [b,a,\overset{2^n}\ldots,a]&\in \gamma_{i_2}(G)\backslash \gamma_{i_2+1}(G)\\
  c&\in \gamma_{i_3}(G)\backslash \gamma_{i_3+1}(G),
\end{align*}
and set
\[
\ell:=\min \{i_1,i_2,i_3+1\}.
\]
In other words, the integer~$\ell$ is such that $[(ab)^{2^n},a]\in\gamma_{\ell}(G)\backslash\gamma_{\ell+1}(G)$.

Now
\[
(a^{1+2\delta}bw)^{2^n}= [bw,a^{1+2\delta},\overset{2^{n-1} -1}\ldots,a^{1+2\delta}]^2 [bw,a^{1+2\delta},\overset{2^n -1}\ldots,a^{1+2\delta}] c d,
\]
where $d\in\gamma_{i_3+1}(G)$ is obtained as follows: starting with the commutator $c$, which is a product of commutators in $b$ and $a$ of weight at least $2^n$, with weight exactly 2 in~$b$; compare \cite[Prop.~1.1.32(i)]{McKay}, we have that 
$cd=\widetilde{c}$, where $\widetilde{c}$ is obtained from $c$ by replacing all occurrences of $b$ with $bw$. Here we have also used the standard identities
$[xy,z]=[x,z]^y[y,z]$ and $[x,yz]=[x,z][x,y]^z$.

Next, using  these standard identities
together with Lemmata~\ref{lem:nilpotent} and~\ref{lem:cyclic}, we have
\begin{align*}
    (a^{1+2\delta}bw)^{2^n}
    &\equiv [bw,a^{1+2\delta},\overset{2^{n-1} -1}\ldots,a^{1+2\delta}]^2 [bw,a^{1+2\delta},\overset{2^n -1}\ldots,a^{1+2\delta}] c \pmod {\gamma_{\ell}(G)}\\
    &\equiv [b,a,\overset{2^{n-1} -1}\ldots,a]^2 [b,a,\overset{2^n -1}\ldots,a] c \pmod {\gamma_{\ell}(G)}.
\end{align*}

Therefore
\[
(a^{1+2\delta}bw)^{2^n} \equiv (ab)^{2^n} \pmod {\gamma_{\ell}(G)}.
\]
In other words, we have
$(a^{1+2\delta}bw)^{2^n} = (ab)^{2^n} v$ for $v\in \gamma_{\ell}(G)$. By Lemma~\ref{lem:all-conjugate}, we have that $v=[(ab)^{2^n},g]$ for some $g\in G$, and hence we are done.
\end{proof}

\subsection*{Acknowledgements}
We thank G.\,A.~Fern\'{a}ndez-Alcober for his valuable feedback and helpful suggestions. We also thank the referee for the helpful comments. The first author is supported by the Spanish Government, grant PID2020-117281GB-I00, partly with FEDER funds.
The first and the second authors are supported by the Basque Government, grant IT974-16.
  The first author is also supported by the National Group for Algebraic and Geometric Structures, and their
Applications (GNSAGA - INdAM).
 The third author acknowledges support from EPSRC, grant EP/T005068/1, and from the Lincoln Institute of Advanced Studies. 
 She also thanks the University of the Basque Country for its hospitality.
The first and second authors would like to thank the University of Lincoln for its hospitality while this research was conducted.

\end{document}